\documentclass[11pt]{amsart}
\usepackage{amsmath,amssymb,latexsym}
\usepackage[english]{babel}
\newtheorem{theorem}{Theorem}[section]
\newtheorem{corollary}[theorem]{Corollary}
\newtheorem{lemma}[theorem]{Lemma}

\theoremstyle{definition}
\newtheorem{definition}[theorem]{Definition}

\newtheorem{remark}[theorem]{Remark}

\numberwithin{equation}{section}
\newcommand{\R}{{\mathbb R}}

\title[Global $L^\infty$ and decay estimate]{Global $L^\infty$ and decay estimate for fractional $p$-Laplacian equations in $D^{s,p}(\R^N)$}

\author[S. Carl]{Siegfried Carl}

\address[S. Carl]{Institut f\"ur Mathematik,  Martin-Luther-Universit\"at Halle-Wittenberg,
D-06099 Halle, Germany}
\email{\tt siegfried.carl@mathematik.uni-halle.de}

\author[K. Perera]{Kanishka Perera}
\address[K. Perera]{Department of Mathematical Sciences, Florida Institute of Technology,
150 W University Blvd, Melbourne, FL 32901, USA}
\email{\\ kperera@fit.edu}

\author[H. Tehrani]{Hossein Tehrani}
\address[H. Tehrani]{Department of Mathematical Sciences,  University of Nevada Las Vegas Box 454020, USA}
\email{\tt tehranih@unlv.nevada.edu}



\date{\today}

\begin{document}

\begin{abstract}
In this paper we  present a new global $L^\infty$-estimate for solutions $u\in D^{s,p}(\R^N)$ of the fractional $p$-Laplacian equation
$$
u\in D^{s,p}(\R^N): (-\Delta_p)^s u=f(x,u) \quad\mbox{in }\R^N,
$$
of the form
$$
\|u\|_{\infty}\le C \Phi(\|u\|_{\beta})
$$
for some  $\beta> p$, where $\Phi: \R^+\to \R^+$ is a data independent function with $\lim_{s\to 0^+}\Phi(s)=0$.
The  obtained $L^\infty$-estimate is used to prove a decay estimate based on pointwise estimates in terms of nonlinear Wolff potentials.
Taking advantage of both the $L^\infty$ and decay estimate we prove a Brezis-Nirenberg type result regarding $D^{s,2}(\R^N)$ versus $C_b\left(\R^N, 1+|x|^{N-2s}\right)$ local minimizers.
\end{abstract}

\maketitle

%%%%%%%%%%%%%%%%%%%%%%%%%%%%%%%%%%%%%%%%%%%%%%%%%%%%
%%%%%%%%%%%%%SECTION1%%%%%%%%%%%%%%%%%%%%%%%%%%%%%%%%%%%

\section{Introduction and Main Results}\label{S1}
\noindent
Let  $s\in (0,1)$ and assume throughout $2\leq p < \frac{N}{s}$. Let $ D^{s,p}(\R^N)$ be the homogeneous fractional Sobolev space which is the completion of $C_c^{\infty}(\R^N)$ with respect to the Gagliardo norm $[\cdot]_{s,p}$ given by
\begin{equation}\label{G-101}
[u]_{s,p}=\left(\int_{\R^N}\int_{\R^N}\frac{|u(x)-u(y)|^p}{|x-y|^{N+sp}} dxdy\right)^{\frac{1}{p}}.
\end{equation}

\smallskip

\noindent Recalling the fractional Sobolev-Gagliardo inequality of the form
\begin{equation}\label{G-102}
S_{s,p}\|u\|_{L^{p^*}(\R^N)}^p\le [u]_{s,p}^p,\quad\forall\ u\in C_c^\infty(\R^N),
\end{equation}
with $S_{s,p}>0$ some constant (see e.g. \cite{BGV-2021} or \cite{Mazya}),  the completion $D^{s,p}(\R^N)$ can be identified with the following function space

\begin{equation}\label{G-103}
D^{s,p}(\R^N)=\{u\in L^{p^*}(\R^N): [u]_{s,p}<\infty\},
\end{equation}
where $p^*=\frac{Np}{N-sp}$ denotes the fractional critical Sobolev exponent. The space $(D^{s,p}(\R^N), [\cdot]_{s,p})$ is a uniformly convex Banach space and thus reflexive, see  \cite{BGV-2021}).

Consider the quasilinear fractional elliptic equation in $\R^N$
\begin{equation}\label{G-106}
u\in D^{s,p}(\R^N): (-\Delta_p)^s u=f(x,u),
\end{equation}
where the  fractional $p$-Laplacian $(-\Delta_p)^s$ is the nonlinear nonlocal operator
formally defined by
\begin{equation}\label{G-105}
(-\Delta_p)^su(x)=2\lim_{\varepsilon\searrow 0}\int_{\R^N\setminus B_{\varepsilon}(x)}\frac{|u(x)-u(y)|^{p-2}(u(x)-u(y))}{|x-y|^{N+sp}}\,dy,\quad x\in\R^N.
\end{equation}
We denote the standard norms of the Lebesgue spaces  $L^r(\R^N)$   by $\|\cdot \|_{r}$,  $1\le r\le \infty$,  and  use $C$, to denote a constant whose exact value is immaterial and may change from line to line. To indicate the dependence of the constant on the data, we write $C=C(a,b,\cdot,\cdot,\cdot)$ with the understanding that this dependence is increasing in its variables.

Our main goals of this paper are three-fold.

First, we  prove a new global $L^\infty$-estimate for weak solutions of (\ref{G-106}),
where the nonlinear right-hand side $f: \R^N\times\R\to \R$ satisfies
\begin{equation}\label{G-107}
|f(x,t)|\le b_2(x)|t|^{p-1}+b_3(x),
\end{equation}
with coefficient functions $b_2$ and $b_3$ that are supposed to satisfy
the following hypotheses:
\begin{itemize}
\item[(H1)] $b_2\in L^{\frac{q}{p}}(\R^N)$, for some $q>\frac{N}{s}$;
\item [(H2)] $b_3\in L^{\frac{\beta}{p-1}}(\R^N)\cap L^\infty(\R^N)$, for some $\beta > p$.
\end{itemize}
We recall the notion of solution for (\ref{G-106}).
\begin{definition}\label{G-D101}
A function $u\in D^{s,p}(\R^N)$ is a weak solution of equation (\ref{G-106}) if
\begin{eqnarray*}
\langle (-\Delta_p)^su, v\rangle &:=&\int_{\R^N}\int_{\R^N}\frac{|u(x)-u(y)|^{p-2}(u(x)-u(y))(v(x)-v(y))}{|x-y|^{N+sp}} dxdy \\ &=&\int_{\R^N}f(x,u)v dx,\quad\forall\ v\in D^{s,p}(\R^N).
\end{eqnarray*}
\end{definition}
Our first main result reads as follows:
\begin{theorem}\label{G-T101}
Assume hypotheses  (H1)--(H2), and let the growth condition (\ref{G-107}) be satisfied. Let $u\in D^{s,p}(\R^N)$ be a solution of the fractional $p$-Laplacian equation (\ref{G-106}), and suppose $u\in L^\beta(\R^N)$ for some $\beta>p$. Then there exists $\theta_0=\theta_0(p,\beta,N,s)$ with $0<\theta_0\le 1$ such that
\begin{equation}\label{G-108}
\|u\|_{\infty}\le C \max\left\{\|u\|_{\beta}, \|u\|_{\beta}^{\theta_0}\right\},
\end{equation}
where $C=C\left(p,q,N,s, \|b_2\|_{\frac{q}{p}}, \|b_3\|\right)$, with $\|b_3\|:=\|b_3\|_{\frac{\beta}{p-1}}+\|b_3\|_{\infty}$.
\end{theorem}
Second, consider the fractional $p$-Laplacian equation
\begin{equation}\label{101}
(-\Delta_p)^su=g(x,u)\quad\mbox{in }\R^N
\end{equation}
and assume the hypotheses:
\begin{itemize}
\item[(Hg)] $g: \R^N\times \R\to \R$ is a Cararth\'eodory function satisfying the growth condition
$$
|g(x,t)|\le |a(x)|\left(1+|t|^{\gamma-1}\right),\quad 1\le \gamma<p^*,
$$
\end{itemize}
where the coefficient $a: \R^N\to\R$ is measurable and satisfies the following decay condition:
\begin{itemize}
\item[(Ha)]
\begin{equation}\label{102}
 |a(x)|\le c_aw(x), \ \mbox{ where }\ w(x)=\frac{1}{1+|x|^{N+\alpha}},\ \ x\in\R^N
\end{equation}
for some $\alpha>0$ and $c_a\ge 0$
\end{itemize}

\smallskip

Our second main result, which strongly relies on Theorem \ref{G-T101}, reads as follows:
\begin{theorem}\label{D-T101}
Assume (Hg)--(Ha). If $u\in D^{s,p}(\R^N)$ is a  solution of (\ref{101}), then a decay estimate of the following form holds:
\begin{equation}\label{103}
|u(x)|\le C |x|^{-\frac{N-sp}{p-1}}, \ \forall\ x: |x|\ge 1,
\end{equation}
where $C=C(N,p,s,\|a\|, [u]_{s,p})$ with $\|a\|:=\|a\|_1+\|a\|_{\infty}$.
\end{theorem}

\smallskip

\noindent Regarding Theorem \ref{G-T101} and Theorem \ref{D-T101} a few comments are in order.

\smallskip
\begin{remark}\label{G-R101}

\begin{itemize}
\item[(i)] In view of  (\ref{G-103}), the assumption $u\in L^\beta(\R^N)$ for some $\beta>p$ for a solution in $D^{s,p}(\R^N)$ is trivially satisfied, since $u\in L^{p^*}(\R^N)$ and $p^*>p$. Theorem \ref{G-T101} holds likewise also for solutions $u$ from the fractional Sobolev space $W^{s,p}(\R^N)$ defined by
$$
W^{s,p}(\R^N)=\{u\in L^p(\R^N): [u]_{s,p}<\infty\}
$$
endowed with the norm
$$
\|u\|_{s,p}=\left(\|u\|_{L^p(\R^N)}^p+[u]_{s,p}^p\right)^{\frac{1}{p}},
$$
since $W^{s,p}(\R^N)\subset L^p(\R^N)\cap L^{p^*}(\R^N)$. We note that the inclusion $W^{s,p}(\R^N)\subset  L^{p^*}(\R^N)$ follows from the fractional  Sobolev-Gagliardo inequality (\ref{G-102}) and the density of $C_c^\infty(\R^N)$ in $W^{s,p}(\R^N)$.

\item[(ii)] A priori $L^\infty$-estimates on bounded domains $\Omega$  for solutions of fractional elliptic equations in the fractional Sobolev space $W^{s,p}(\Omega)$, as well as  for solutions in the fractional Sobolev space with variable exponents $W_0^{s,p(\cdot)}(\Omega)$ have been obtained, e.g., in \cite{BDMQ-2018,  L-2019, MPL-2022}, and \cite{HK-2019}, respectively. In particular, in
    \cite[Section 4]{HK-2019}, the authors obtained an estimate (when considered constant exponents) of the form
$$
\|u\|_{\infty}\le C\max\left\{\|u\|^{\tau_1}_{L^{\tilde{q}}(\Omega)}, \|u\|^{\tau_2}_{L^{\tilde{q}}}(\Omega)\right\},
$$
where $\tilde{q}:=\max\{p, q\}$ with $q<p^*$, which formally is similar to the result in Theorem \ref{G-T101}. However, in \cite{HK-2019} the boundedness of the domain plays an important role to get the above $L^\infty$-estimates. The proof of  Theorem \ref{G-T101} for solutions  of the fractional $p$-Laplacian equation in all $\R^N$, which  is carried out within the framework of the homogeneous fractional Sobolev space $D^{s,p}(\R^N)$,  is very much different and  applies likewise to any domain $\Omega\in \R^N$,  see Remark \ref{G-R201}.

\item[(iii)] To the best of our knowledge the $L^\infty$-estimate (\ref{G-108}) in all $\R^N$  is new, which allows to control the   $\|u\|_{\infty}$ norm by an integral norm  $\|u\|_{\beta}$, where $\beta>p$, and thus in particular, by $\|u\|_{p^*}$. As the solution $u$ is in the  homogeneous fractional Sobolev space $D^{s,p}(\R^N)$, finally the $\|u\|_{\infty}$ norm can be controlled by $[u]_{s,p}$. Hence if  $[u]_{s,p}\to 0$, then $\|u\|_{\infty}\to 0$.  The $L^\infty$-estimate together with Wolff potential estimates are the essential ingredients in the proof of Theorem \ref{D-T101}.
\end{itemize}
\end{remark}

\smallskip

Third, based on the results of Theorem \ref{G-T101} and Theorem \ref{D-T101},  a Brezis-Nirenberg type result   regarding $D^{s,2}(\R^N)$ versus $C_b\left(\R^N, 1+|x|^{N-2s}\right)$ local minimizers for the fractional Laplacian equation
$$
u\in D^{s,2}(\R^N): (-\Delta)^su=a(x) g(u)
$$
is proved in Theorem \ref{T501}, which significantly improves  the result of \cite{Ambrosio-23}, see Section \ref{S5}.

\smallskip

\noindent The outline of the paper is as follows: In Section \ref{S2} we  present the new global $L^\infty$-estimate (\ref{G-108}) for solutions $u\in D^{s,p}(\R^N)$ of the fractional $p$-Laplacian equation (\ref{G-106}) that has the form
$$
\|u\|_{\infty}\le C \Phi(\|u\|_{\beta})
$$
for some  $\beta> p$, where $\Phi: \R^+\to \R^+$ is a data independent function with $\lim_{s\to 0^+}\Phi(s)=0$, and prove  Theorem \ref{G-T101} along with some useful applications. In  Section \ref{S3} we are going to prove pointwise estimates of fractional $p$-Laplacian supersolutions in terms of nonlinear Wolff potentials and tail functions in all $\R^N$, and in Section \ref{S4} we present the proof of Theorem \ref{D-T101}, which relies on both the a priori  $L^\infty$-estimate of Theorem \ref{G-T101} and the pointwise estimates of Section \ref{S3}. In Section \ref{S5},
taking advantage of both the $L^\infty$ and decay estimates of the preceding sections,  we prove a Brezis-Nirenberg type result regarding $D^{s,2}(\R^N)$ versus $C_b\left(\R^N, 1+|x|^{N-2s}\right)$ local minimizers.

%%%%%%%%%%%%%%%%%%%%%%%%%%%%%%%%%%%%%%%%%%%%%%%%%%%%
%%%%%%%%%%%%%SECTION2%%%%%%%%%%%%%%%%%%%%%%%%%%%%%%%%%%%
\section{$L^\infty$ estimate}\label{S2}
%%%%%%%%%%%%%%%%%%%%%%%%%%%%%%%%%%%%%%%%%%%%%%%%%%%%%%%%%%%%%%%%%%%%%%%%%
\subsection{Preliminary results}\label{S21}
%%%%%%%%%%%%%%%%%%%%%%%%%%%%%%%%%%%%%%%%%%%%%%%%%%%%
Throughout this section we use the notation
$$
a^{1+\gamma}:=a |a|^\gamma\quad\mbox{for any $a\in\R$ and $\gamma>0$}.
$$
On $\R$ we introduce the function $a\mapsto  a_L^{1+\gamma}$ with $\gamma>0$ and $L>0$ defined by
$$
a_L^{1+\gamma}=a\min\{|a|,L\}^\gamma.
$$
Note that this definition satisfies the following simple properties, which will be freely used throughout the presentation below.
$$
|a_L^{1+\gamma}|=|a|_L^{1+\gamma},\,\quad\quad |a_L^{1+\gamma}|\leq |a|^{1+\gamma},\,  \quad\quad  |a|_L^{(1+\gamma_1)(1+\gamma_2)}\leq (|a|_L^{1+\gamma_1})^{1+\gamma_2}.
$$
Next we recall \cite[Lemma A.2]{BP-2016} (also see lemma 2.3 in \cite{IMS-2020}) from which for $g(t)= t^{1+p\gamma}$ one easily deduces the following:
%\
\begin{lemma}\label{G-C201}
For any $a,b\in\R$ and $\gamma>0$,  the following inequality holds:
\begin{equation}\label{G-201}
\left|a^{1+\gamma}-b^{1+\gamma}\right|^p\le C(1+p\gamma)^{p-1}(a-b)^{p-1}\left(a^{1+p\gamma}-b^{1+p\gamma}\right),
\end{equation}
where $C=C(p)$.
\end{lemma}
By making use of (\ref{G-201}) and  carefully considering different cases, we are going to prove that the following analogous inequality holds true for the function $a\mapsto a_L^{1+\gamma}$ as well.
\begin{lemma}\label{G-L202}
For any $a,b\in\R \,$ and $\gamma>0$, the following inequality holds:
\begin{equation}\label{G-202}
\left|a_L^{1+\gamma}-b_L^{1+\gamma}\right|^p\le C(1+p\gamma)^{p-1}(a-b)^{p-1}\left(a_L^{1+p\gamma}-b_L^{1+p\gamma}\right),
\end{equation}
where $C=C(p)$.
\end{lemma}
\begin{proof}
We first note that, without loss of generality, we may assume $b<a$, from which it easily follows that  $b_L^{1+\gamma}\le a_L^{1+\gamma}$, for any $\gamma, L >0$. Therefore we will prove
\begin{equation}\label{G-203}
\left(a_L^{1+\gamma}-b_L^{1+\gamma}\right)^p\le C(1+p\gamma)^{p-1}(a-b)^{p-1}\left(a_L^{1+p\gamma}-b_L^{1+p\gamma}\right)
\end{equation}
by considering the following cases.

\smallskip

\noindent Case 1: $|a|, |b|< L$. Then $a_L^{1+\gamma}=a^{1+\gamma}$ and $b_L^{1+\gamma}=b^{1+\gamma}$, thus (\ref{G-203}) is true due to (\ref{G-201}).

\smallskip

\noindent Case 2: $|b|< L, \, L<a$. Then $a_L^{1+\gamma}=aL^\gamma$ and $b_L^{1+\gamma}=b^{1+\gamma}$, and
\begin{eqnarray*}
\left(a_L^{1+\gamma}-b_L^{1+\gamma}\right)^p&=&\left(aL^{\gamma}-b^{1+\gamma}\right)^p=\left(aL^{\gamma}-L^{1+\gamma}+L^{1+\gamma}-b^{1+\gamma}\right)^p\\
&\le & 2^{p-1}\left[\left(aL^{\gamma}-L^{1+\gamma}\right)^p+\left(L^{1+\gamma}-b|b|^{\gamma}\right)^p\right]\\
&\le & 2^{p-1} (a-L)(a-L)^{p-1}L^{p\gamma} \\
&& + 2^{p-1}C(p)(1+p\gamma)^{p-1}(L-b)^{p-1}\left(L^{1+p\gamma}-b^{1+p\gamma}\right) \\
&\le & 2^{p-1} (a-L)(a-b)^{p-1}L^{p\gamma}\\
&& +2^{p-1}C(p)(1+p\gamma)^{p-1}(a-b)^{p-1}\left(L^{1+p\gamma}-b^{1+p\gamma}\right) \\
&\le & C(p)(1+p\gamma)^{p-1}(a-b)^{p-1}\left[(a-L)L^{p\gamma}+L^{1+p\gamma}-b^{1+p\gamma}\right]\\
&\le & C(p)(1+p\gamma)^{p-1}(a-b)^{p-1}\left(aL^{p\gamma}-b^{1+p\gamma}\right),
\end{eqnarray*}
which proves (\ref{G-203}) in this case.

\smallskip

\noindent Case 3: $|b|>L,$ $|a|<L$, and recall $a>b$, which means $b<-L<a<L$.

Then we have $a_L^{1+\gamma}=a^{1+\gamma}$ and $b_L^{1+\gamma}=bL^\gamma$, which yields
\begin{eqnarray*}
\left(a^{1+\gamma}-bL^\gamma\right)^p&=&\left(a^{1+\gamma}-(-L)L^\gamma-LL^\gamma-bL^\gamma\right)^p\\
&\le & 2^{p-1}\left[\left(a^{1+\gamma}-(-L)L^\gamma\right)^p+\left(-LL^\gamma-bL^\gamma\right)^p\right]\\
&\le & C(p)(1+p\gamma)^{p-1}(a-(-L))^{p-1}\left(a^{1+p\gamma}-(-L)^{1+p\gamma}\right)\\
&& +C(p)(-L-b)(-L-b)^{p-1}L^{p\gamma}\\
&\le & C(p)(1+p\gamma)^{p-1}(a-b)^{p-1}\left[ a^{1+p\gamma}-(-L)^{1+p\gamma}+(-L-b)L^{p\gamma}\right]\\
&\le & C(p)(1+p\gamma)^{p-1}(a-b)^{p-1}\left[ a^{1+p\gamma}- bL^{p\gamma}\right]\\
&\le & C(p)(1+p\gamma)^{p-1}(a-b)^{p-1}\left[ a_L^{1+p\gamma}- b_L^{1+p\gamma}\right],
\end{eqnarray*}
which is (\ref{G-203}).

\smallskip

\noindent Case 4: $|b|>L,$ $|a|>L$ and  $a>b$.

Then we have $a_L^{1+\gamma}=a L^{\gamma}$ and $b_L^{1+\gamma}=bL^\gamma$, and for the left-hand side of (\ref{G-203}) we get
$$
\left(a_L^{1+\gamma}-b_L^{1+\gamma}\right)^p= \left((a-b)L^\gamma\right)^p=(a-b)^pL^{p\gamma}.
$$
The right-hand side of (\ref{G-203}) yields
$$
C(p)(1+p\gamma)^{p-1}(a-b)^{p-1}\left[aL^{p\gamma}-bL^{p\gamma}\right]=C(p)(1+p\gamma)^{p-1}(a-b)^pL^{p\gamma},
$$
which proves  (\ref{G-203}) also in this case, since $C(p)(1+p\gamma)^{p-1}\ge 1$.
\end{proof}
\begin{lemma}\label{G-L203}
For all $a, b,c,d\in\R$ such that $a-b=c-d$ and $\gamma>0$, the following inequality holds:
\begin{equation}\label{G-204}
\left|a_L^{1+\gamma}-b_L^{1+\gamma}\right|^p\le C(1+p\gamma)^{p-1}(c^{p-1}-d^{p-1})\left(a_L^{1+p\gamma}-b_L^{1+p\gamma}\right),
\end{equation}
where $C=C(p)$
\end{lemma}
\begin{proof}
Without loss of generality we may assume $a\ge b$ which implies $c\ge d$ due to $a-b=c-d$, and prove
\begin{equation}\label{G-205}
\left(a_L^{1+\gamma}-b_L^{1+\gamma}\right)^p\le C(1+p\gamma)^{p-1}(c^{p-1}-d^{p-1})\left(a_L^{1+p\gamma}-b_L^{1+p\gamma}\right).
\end{equation}
%
%From \cite[p.6]{IMS-2020},
Since $p\ge 2$, we have the inequality
\begin{eqnarray*}
\left(c^{p-1}-d^{p-1}\right)(c-d)&=&\left(c|c|^{p-2}-d|d|^{p-2}\right)(c-d)\ge\frac{1}{2^{p-2}}(c-d)^p,
\end{eqnarray*}
and thus
$$
c^{p-1}-d^{p-1} \ge \frac{1}{2^{p-2}}(c-d)^{p-1},
$$
which yields
$$
(a-b)^{p-1}=(c-d)^{p-1}\le 2^{p-2}\left(c^{p-1}-d^{p-1}\right).
$$
Hence (\ref{G-205}) follows from (\ref{G-203}).
\end{proof}
\begin{lemma}\label{G-L204}
For any $u, v\in D^{s,p}(\R^N)$  the following inequality holds:
\begin{equation}\label{G-206}
\left[(u-v)_L^{1+\gamma}\right]^p_{s,p}\le C(p)(1+p\gamma)^{p-1}\left\langle (-\Delta_p)^su-(-\Delta_p)^sv,(u-v)_L^{1+p\gamma}\right\rangle.
\end{equation}
\end{lemma}
\begin{proof} We make use of inequality (\ref{G-204}) by setting  $a=u(x)-v(x),$ $b=u(y)-v(y),$ $c=u(x)-u(y),$ $d=v(x)-v(y),$ and obtain by integration
\begin{eqnarray*}
&&\left[(u-v)_L^{1+\gamma}\right]^p_{s,p}=\iint \left|(u(x)-v(x))_L^{1+\gamma}-(u(y)-v(y))_L^{1+\gamma}\right|^p\,d\mu\\
&&\le C(p,\gamma)\iint \left((u(x)-u(y))^{p-1}-(v(x)-v(y))^{p-1}\right)\\
&& \qquad\qquad \times\left((u(x)-v(x))_L^{1+p\gamma}-(u(y)-v(y))_L^{1+p\gamma}\right)d\mu\\
&& \le C(p,\gamma)\left\langle (-\Delta_p)^su-(-\Delta_p)^sv,(u-v)_L^{1+p\gamma}\right\rangle,
\end{eqnarray*}
where $ C(p,\gamma)= C(p)(1+p\gamma)^{p-1}$ and $  d\mu:= \frac{1}{|x-y|^{N+sp}}dxdy$.
\end{proof}
\begin{corollary}\label{G-C202}
If $u\in D^{s,p}(\R^N)$, then  $u_L^{1+\gamma}\in D^{s,p}(\R^N)$.
\end{corollary}
\begin{proof}
We use inequality (\ref{G-202}) with $a=u(x)$ and $b=u(y)$ and get
\begin{equation}\label{G-207}
\left|u_L^{1+\gamma}(x)-u_L^{1+\gamma}(y)\right|^p\le C(1+p\gamma)^{p-1}(u(x)-u(y))^{p-1}\left(u_L^{1+p\gamma}(x)-u_L^{1+p\gamma}(y)\right).
\end{equation}
Taking note of the fact that the function $a\mapsto a_L^{1+p\gamma}:=a\min\{|a|,L\}^{p\gamma}$ is increasing and linear for $|a|\ge L$,  and  $a\mapsto a|a|^{p\gamma}$ is Lipschitz within the interval $[-L,L]$, the function $a\mapsto a_L^{1+p\gamma}$  is globally Lipschitz on $\R$, that is, we have
$$
\left| a_L^{1+p\gamma}- b_L^{1+p\gamma}\right|\le C(L,p,\gamma)|a-b|.
$$
With the last inequality from (\ref{G-207}) we obtain
\begin{equation}\label{G-208}
\left|u_L^{1+\gamma}(x)-u_L^{1+\gamma}(y)\right|^p\le C(L,p,\gamma)|u(x)-u(y)|^p.
\end{equation}
Integrating (\ref{G-208}) yields
$$
\left[u_L^{1+\gamma}\right]_{s,p}^p=\iint\left|u_L^{1+\gamma}(x)-u_L^{1+\gamma}(y)\right|^p\,d\mu\le C[u]_{s,p}^p,
$$
where $C=C(L,p,\gamma)$ and as above $d\mu=\frac{1}{|x-y|^{N+sp}}dxdy$, which completes the proof as clearly $u_L^{1+\gamma}$ is in $L^{p^*}(\R^N)$.
\end{proof}
%

%%%%%%%%%%%%%%%%%%%%%%%%%%%%%%%%%%%%%%%%%%%%%%%%%%%%
%%%%%%%%%%%%%Subsection22%%%%%%%%%%%%%%%%%%%%%%%%%%%%%%%%%%%
\subsection{Proof of Theorem \ref{G-T101}}\label{S22}
%%%%%%%%%%%%%%%%%%%%%%%%%%%%%%%%%%%%%%%%%%%%%%%%%%%%
Let $u\in D^{s,p}(\R^N)$ be a solution of the fractional $p$-Laplacian equation (\ref{G-106}) which we recall here
\begin{equation}\label{G-301}
u\in D^{s,p}(\R^N): (-\Delta_p)^s u=f(x,u) \quad\mbox{in }\R^N,
\end{equation}
where the nonlinear right-hand side $f: \R^N\times\R\to \R$ satisfies the growth (\ref{G-107}), that is,
\begin{equation}\label{G-302}
|f(x,t)|\le b_2(x)|t|^{p-1}+b_3(x).
\end{equation}
Under hypotheses (H1)--(H2), our goal is to prove the global $L^\infty$-estimate (\ref{G-108}).

Let $\alpha=\beta-(p-1)>1$, $L>1$. In view of Corollary \ref{G-C202} we have
$$
u_L^\alpha=u\min\{|u|,L\}^{\alpha-1}\in D^{s,p}(\R^N).
$$
Applying inequality (\ref{G-202}) with $\gamma=\frac{\alpha-1}{p}$, and $a=u(x)$, $b=u(y)$ we get
$$
\left|u_L^{\frac{p+\alpha-1}{p}}(x)-u_L^{\frac{p+\alpha-1}{p}}(y)\right|\le C(p)\alpha^{p-1}(u(x)-u(y))^{p-1}\left(u_L^\alpha(x)-u_L^\alpha(y)\right).
$$
Integrating the last inequality by $\R^N\times\R^N$, using the measure $d\mu=\frac{1}{|x-y|^{N+sp}}dxdy$ yields
$$
\left[u_L^{\frac{p+\alpha-1}{p}}\right]^p_{s,p}\le C(p)\alpha^{p-1}\left\langle (-\Delta_p)^su, u_L^\alpha\right\rangle=C(p)\alpha^{p-1}\int_{\R^N}f(x,u)u_L^\alpha\,dx.
$$
Since $u_L^{\frac{p+\alpha-1}{p}}\in D^{s,p}(\R^N)$, by the  fractional Sobolev-Gagliardo inequality we have
\begin{equation}\label{G-303}
\left\|u_L^{\frac{p+\alpha-1}{p}}\right\|_{p^*}^p\le C(s,p) \left[u_L^{\frac{p+\alpha-1}{p}}\right]^p_{s,p},
\end{equation}
hence from the last two inequalities along with (\ref{G-302}) we obtain
\begin{eqnarray}
\left\|u_L^{\frac{p+\alpha-1}{p}}\right\|_{p^*}^p&\le& C(s,p)\alpha^{p-1} \int_{\R^N}b_2(x)|u|^p\min\{|u|,L\}^{\alpha-1}\,dx\nonumber\\
&+&C(s,p)\alpha^{p-1}\int_{\R^N}b_3(x)|u_L^\alpha|\,dx.\label{G-304}
\end{eqnarray}
Next, let us estimate the integrals on the right-hand side of (\ref{G-304}). We recall that by hypotheses $sp<N$ and $N<sq$, and thus
$$
p< \frac{pq}{q-p}<p^*=\frac{Np}{N-sp}.
$$
Since $\left|u_L^{\frac{p+\alpha-1}{p}}\right|^p\le |u|^{p+\alpha-1}=|u|^\beta$, it follows that $\left|u_L^{\frac{p+\alpha-1}{p}}\right|\in L^p(\R^N)$. Moreover, by (\ref{G-303}) we have also $\left|u_L^{\frac{p+\alpha-1}{p}}\right|\in L^{p^*}(\R^N)$. Hence, by interpolation it follows that $\left|u_L^{\frac{p+\alpha-1}{p}}\right|\in L^{\frac{pq}{q-p}}(\R^N)$, which along with Young's inequality allows the following estimate for the first integral on the right-hand side of (\ref{G-304}):
\begin{eqnarray*}
I_1&=&\int_{\R^N}b_2(x)|u|^p\min\{|u|,L\}^{\alpha-1}\,dx=\int_{\R^N}b_2(x)\left|u_L^{\frac{p+\alpha-1}{p}}\right|^p\,dx\\
&\le& \|b_2\|_{\frac{q}{p}}\left\|u_L^{\frac{p+\alpha-1}{p}}\right\|^p_{\frac{pq}{q-p}}\\
&\le& \|b_2\|_{\frac{q}{p}}\left(\varepsilon\left\|u_L^{\frac{p+\alpha-1}{p}}\right\|_{p^*}^p+
C(\varepsilon)\left\|u_L^{\frac{p+\alpha-1}{p}}\right\|_{p}^p\right),
\end{eqnarray*}
for any $\varepsilon>0$, where $C(\varepsilon)=\varepsilon^{-\frac{N}{qs-N}}$. Taking, in particular,
$$
\varepsilon=\frac12\frac{1}{C(p,s)\alpha^{p-1}}\frac{1}{\|b_2\|_{\frac{q}{p}}},
$$
then from (\ref{G-304}) we get by taking into account that $\left|u_L^{\frac{p+\alpha-1}{p}}\right|^p\le |u|^{p+\alpha-1}$
\begin{eqnarray}
\left\|u_L^{\frac{p+\alpha-1}{p}}\right\|_{p^*}^p&\le & \tilde{C}(p,s)\alpha^{(p-1)\frac{qs}{qs-N}}\|b_2\|_{\frac{q}{p}}^{\frac{qs}{qs-N}}\int_{\R^N}|u|^{p+\alpha-1}dx\nonumber\\
&& +C(s,p)\alpha^{p-1}\int_{\R^N}b_3(x)|u_L^\alpha|\,dx \label{G-305}
\end{eqnarray}
Now let us also estimate the second integral on the right-hand side of (\ref{G-304}) or (\ref{G-305}) and recall $\alpha=\beta-(p-1)>1$, that is, $\beta=p+\alpha-1$.
\begin{eqnarray}\label{G-306}
I_2=\int_{\R^N}b_3(x)|u_L^\alpha|\,dx&\le &\int_{\R^N}b_3(x)|u|^{\alpha}\,dx\nonumber\\
&\le & \|b_3\|_{\frac{\beta}{p-1}}\left(\int_{\R^N}|u|^\beta dx\right)^{\frac{\alpha}{\beta}}
\end{eqnarray}
In view of
$$
\left(|u_L(x)|^{\frac{\beta}{p}}\right)^{p^{*}}\geq |u_L(x)|^{\frac{\beta}{p}p^{*}}=|u_L(x)|^{\frac{\beta N}{N-ps}}
$$
we conclude (note: $\beta=p+\alpha-1$)
$$
\left\|u_L^{\frac{p+\alpha-1}{p}}\right\|_{p^*}^p=\left\|u_L^{\frac{\beta}{p}}\right\|_{p^*}^p\ge\|u_L\|_{\frac{N\beta}{N-sp}}^\beta,
$$
hence from (\ref{G-305}) and (\ref{G-306}) we get
\begin{eqnarray*}
\|u_L\|_{\frac{N\beta}{N-sp}}^\beta&\le&\tilde{C}(p,s)\alpha^{(p-1)\frac{qs}{qs-N}}\|b_2\|_{\frac{q}{p}}^{\frac{qs}{qs-N}}\|u\|_\beta^\beta\\
&&+C(s,p)\alpha^{p-1} \|b_3\|_{\frac{\beta}{p-1}}\|u\|_\beta^\alpha.
\end{eqnarray*}
Letting $L\to\infty$, we finally get the following inequality:
\begin{equation}\label{G-307}
\|u\|_{\frac{N}{N-sp}\beta}\le (C \beta)^{\frac{\sigma}{\beta}}\max\left\{\|u\|_\beta, \|u\|_\beta^{1-\frac{p-1}{\beta}}\right\},
\end{equation}
where $C=C(p,q, s,N, \|b_2\|_{\frac{q}{p}},\|b_3\|)$ and $\sigma=\sigma(p,q,s,N)>1$.

Iterating inequality (\ref{G-307})by  taking $\beta_0=\beta,\, \, \beta_k=\chi^k\beta\, $, with $\chi=\frac{N}{N-sp}$, and keeping in mind that we may assume $C\beta \geq 1$,  we obtain
$$ \|u\|_{\chi^k\beta}\leq (C\chi\beta)^{\tilde{\sigma}}\max\left\{\|u\|_{\beta},|u|_{\beta}^{\theta(k)}\right\},\quad\quad k\in \mathbb{Z}^{+}
$$
where $\tilde{\sigma}= \frac{\sigma}{\beta}(1+\sum_{m=1}^{\infty} \frac{m}{\chi^m})$ and
$$\theta(k)=\prod_{i=0}^{k-1} R(i),\quad\quad R(i)=1\, \,  \mbox{ or }\, \, R(i)=1-\frac{p-1}{\chi^i \beta}, \,\, \mbox{ for }\,
\, 0\leq i\leq k-1.
$$
Note that
$$0< \theta_0=\theta_0(p,\beta, N) :=\prod_{i=0}^{\infty}(1-\frac{p-1}{\chi^i \beta})\leq \theta(k)\leq 1
$$
since  $\sum_{i=0}^{\infty} \frac{p-1}{\chi^i \beta}<\infty$. Hence

\begin{equation}\label{G-308}
\|u\|_r\leq C\max\{ \|u\|_{\beta}^{\theta_0}, \, \|u\|_{\beta}\},\quad \forall r \mbox{ with } p\leq r<\infty,
\end{equation}
and with $C$ as in the statement of Theorem \ref{G-T101}, which implies  (\ref{G-108}), completing the proof of the theorem.
\hfill $\Box$
\begin{remark}\label{G-R201}
The $L^\infty$-estimate given in Theorem \ref{G-T101} can be extended in a straightforward way to solutions $u\in D^{s,p}_0(\Omega)$ for any domain $\Omega\subset \R^N$. In this case  $u\in D^{s,p}_0(\Omega)$ is a solution of  the Dirichlet problem
$$
u\in D^{s,p}_0(\Omega): (-\Delta_p)^s u=f(x,u) \quad\mbox{in }\Omega,\quad u=0\quad\mbox{in }\R^N\setminus \Omega,
$$
where $D^{s,p}_0(\Omega)$ denotes the completion of smooth functions with compact support in $\Omega$ with respect to the $[.]_{s,p}$ norm, which due to the fractional Sobolev-Gagliardo inequality allows for the following characterization:
$$
D^{s,p}_0(\Omega)=\{ u\in L^{p^*}(\Omega): [u]_{s,p}<\infty, \quad u=0 \mbox{ in } \R^N\setminus \Omega\}.
$$
\end{remark}
%

%%%%%%%%%%%%%%%%%%%%%%%%%%%%%%%%%%%%%%%%%%%%%%%%%%%%
%%%%%%%%%%%%%SUBSECTION23%%%%%%%%%%%%%%%%%%%%%%%%%%%%%%%%%%%
\subsection{Application}\label{S23}
%%%%%%%%%%%%%%%%%%%%%%%%%%%%%%%%%%%%%%%%%%%%%%%%%%%%%%%%%
%%%%%%%%%%%%%%%%%%%%%%%%%%%%%%%%%%%%%%%%%%%%%%%%%%%%%%%%%%%
Consider the fractional $p$-Laplacian equation (\ref{101}), which we recall here
\begin{equation}\label{G-401}
(-\Delta_p)^su=g(x,u)\quad\mbox{in }\R^N,
\end{equation}
and assume the hypothesis (Hg), where  $a\in L^1(\R^N)\cap L^\infty(\R^N)$ with $\|a\|:=\|a\|_1+\|a\|_{\infty}.$
\begin{corollary}\label{G-C401}
Assume (Hg) and $a\in L^1(\R^N)\cap L^\infty(\R^N)$.  If $u\in D^{s,p}(\R^N)$ is a solution of (\ref{G-401}), then $u$ satisfies a $L^\infty$-estimate of the form
\begin{equation}\label{G-402}
\|u\|_{\infty}\le C\max\{ \|u\|_{p^*}^{\theta_0}, \, \|u\|_{p^*}\},
\end{equation}
where $C=C(p,q,N,s, \|a\|, [u]_{s,p})$ and $\theta_0\in (0,1]$.
\end{corollary}
\begin{proof}
For the right-hand side of (\ref{G-401}) we have the estimate
$$
|g(x,u(x))|\le |a(x)|+|a(x)||u(x)|^{\gamma-p}|u(x)|^{p-1}.
$$
Using the notation of Theorem \ref{G-T101}, we set
\begin{equation}\label{G-403}
b_2(x)=|a(x)||u(x)|^{\gamma-p},\quad\mbox{and } b_3(x)=|a(x)|.
\end{equation}
In order to apply Theorem \ref{G-T101}, we need to verify hypotheses (H1)--(H2) for $b_2$ and $b_3$ given by (\ref{G-403}).
First, note that
$$
\frac{N}{sp}(\gamma-p)<p^*=\frac{Np}{N-sp} \Longleftrightarrow \gamma< \frac{sp^2}{N-sp}+p=p^*.
$$
Thus there exists $q>\frac{N}{s}$ such that $\frac{q}{p}(\gamma-p)< p^*$,
which allows for the estimate
\begin{eqnarray*}
\int_{\R^N}|b_2(x)|^{\frac{q}{p}}\,dx&=& \int_{\R^N}|a(x)|^{\frac{q}{p}}|u(x)|^{(\gamma-p)\frac{q}{p} }\,dx\\
&\le & \|u\|_{p^*}^{(\gamma-p)\frac{q}{p} }\|a\|_r^{\frac{q}{p}},
\end{eqnarray*}
for some $r\in (1,\infty)$. Hence we get
\begin{equation}\label{G-404}
\|b_2\|_{\frac{q}{p}}\le \|a\|_r \|u\|_{p^*}^{\gamma-p}\le  \|a\|_r [u]_{s,p}^{\gamma-p},
\end{equation}
and also
\begin{equation}\label{G-405}
\|b_3\|:=\|b_3\|_{\frac{\beta}{p-1}}+\|b_3\|_\infty\le 2( \|a\|_1+\|a\|_\infty).
\end{equation}
Applying Theorem \ref{G-T101} with $\beta=p^*>p$ yields (\ref{G-402}).
\end{proof}
%
%%%%%%%%%%%%%%%%%%%%%%%%%%%%%%%%%%%%%%%%%%%%%%%%%%%%
%%%%%%%%%%%%% SECTION3%%%%%%%%%%%%%%%%%%%%%%%%%%%%%%%%%%%
\section{Decay estimate}\label{S3}
%%%%%%%%%%%%%%%%%%%%%%%%%%%%%%%%%%%%%%%%%%%%%%%%%%%%%%%%%
%%%%%%%%%%%%%%%%%%%%%%%%%%%%%%%%%%%%%%%%%%%%%%%%%%%%%%%%%%
%%%%%%%%%%%%%SUBSECTION31%%%%%%%%%%%%%%%%%%%%%%%%%%%%%%%%%%%
\subsection{Preliminary results}\label{S31}
%%%%%%%%%%%%%%%%%%%%%%%%%%%%%%%%%%%%%%%%%%%%%%%%%%%%%%%%%%%
%%%%%%%%%%%%%%%%%%%%%%%%%%%%%%%%%%%%%%%%%%%%%%%%%%%%
We first derive some properties of the coefficient $a$.
\begin{lemma}\label{L201}
If $a: \R^N\to\R$ satisfies (Ha),
then $a$ has the following properties:
\begin{itemize}
\item[(A1)] $a\in L^q(\R^N)$ for $1\le q\le \infty$.
\item[(A2)] There exists $C >0$ such that
\begin{equation}\label{201}
|x|^{\frac{N}{\sigma'}}\|a\|_{L^{\sigma}(\mathbb{R}^N\setminus B(0, |x|))}\le C,
\end{equation}
where $\frac{1}{\sigma'}+\frac{1}{\sigma}=1$ for any $\sigma$ with $1<\sigma<\infty$, in particular for  $\sigma> \frac{N}{sp}$. Here $B(0, |x|)$ is the open ball with radius $|x|$.
\end{itemize}
\end{lemma}
\begin{proof}
Clearly, $a\in L^\infty(\mathbb{R}^N)$. As for $a\in L^1(\mathbb{R}^N)$, it follows by the following  estimate ($|x|=r$):
\begin{eqnarray*}
\int_{\mathbb{R}^N} |a(x)|\, dx &\le& c \int_0^\infty \frac{1}{1+r^{N+\alpha}} r^{N-1}\,dr\\
&\le & c\,\int_0^1\frac{1}{1+r^{N+\alpha}} r^{N-1}\,dr+c\,\int_1^\infty r^{-\alpha-1}\,dr <\infty.
\end{eqnarray*}
As for the proof of  $a\in L^q(\R^N)$ for $1<q<\infty$, let $\Omega=\{x\in\R^N: |a(x)|\ge 1\}$.
From $a \in L^1(\mathbb{R}^N)$ we infer that the Lebesgue measure  $|\Omega|<\infty$, and thus we get the estimate
\begin{eqnarray*}
\int_{\mathbb{R}^N}|a|^q\,dx&=&\int_{\Omega}|a|^q\,dx+\int_{\mathbb{R}^N\setminus \Omega}|a|^q\,dx\le \|a\|_{\infty}^q |\Omega|+\int_{\mathbb{R}^N\setminus \Omega}|a|^q\,dx\\
&\le &\|a\|_{\infty}^q |\Omega|+\int_{\{x\in \mathbb{R}^N: 0\le |a(x)|<1\}}|a|^q\,dx\\
&\le &\|a\|_{\infty}^q |\Omega|+\|a\|_1<\infty,
\end{eqnarray*}
which proves (A1).  As for the proof of (\ref{201}) we only need to consider the case that $|x|\ge 1$, since for  $|x|\le 1$ inequality (\ref{201}) is trivially satisfied. Let $|x|\ge 1$ and let $c$ be a generic positive constant.
From (Ha) we have $|a(x)|\le c_a w(x)$ with $w$ as in (\ref{102})). Using spherical coordinates we can estimate as follows
\begin{eqnarray*}
|x|^{\frac{N}{\sigma'}\sigma}\|a\|^{\sigma}_{L^{\sigma}(\mathbb{R}^N\setminus B(0, |x|))}&=& c\,|x|^{\frac{N}{\sigma'}\sigma}\int_{\mathbb{R}^N\setminus B(0, |x|)}  w^\sigma\,dx\\
&\le& c\, |x|^{\frac{N}{\sigma'}\sigma}\int_{|x|}^\infty \Big(\frac{1}{1+r^{N+\alpha}}\Big)^{\sigma} r^{N-1}\,dr\\
&\le & c\,|x|^{\frac{N}{\sigma'}\sigma} \int_{|x|}^\infty r^{-N\sigma-\alpha\sigma+N-1}\,dr\\
&\le & c\,|x|^{\frac{N}{\sigma'}\sigma} |x|^{-N\sigma-\alpha\sigma+N} \\
&\le&  c\,|x|^{-\alpha\sigma}  \le c ,
\end{eqnarray*}
since $\frac{N}{\sigma'}\sigma-N\sigma+N=0$, which proves (\ref{201}).
\end{proof}
Let us recall the following properties of $(-\Delta_p)^s$.
\begin{lemma}\label{L202}
The fractional $p$-Laplacian $(-\Delta_p)^s$ has the following properties:
\begin{itemize}
\item[(i)] $(-\Delta_p)^s: D^{s,p}(\R^N)\to (D^{s,p}(\R^N))^*$ is a bounded, continuous, and strictly monotone operator.
\item[(ii)] $(-\Delta_p)^s$ possesses the following compactness property, called the  (S$_+$) property: if $(u_n)\subset  D^{s,p}(\R^N)$ with $u_n\rightharpoonup u$ (weakly convergent in $D^{s,p}(\R^N))$ and $\limsup_{n\to\infty}\langle (-\Delta_p)^su_n,u_n-u\rangle \le 0$, then $u_n\to u$ (strongly convergent) in $D^{s,p}(\R^N)$.
\end{itemize}
\end{lemma}
\begin{proof}  We may omit the proof of property (i).

To prove the (S$_+$) property (ii), let $u_n\rightharpoonup u$ and $\limsup_{n\to\infty}\langle (-\Delta_p)^su_n,u_n-u\rangle \le 0$. Then with
$\varphi_n(x,y)=\frac{u_n(x)-u_n(y)}{|x-y|^{(N+sp)\frac{1}{p}}}$, $\psi(x,y)=\frac{u(x)-u(y)}{|x-y|^{(N+sp)\frac{1}{p}}}$
and by applying H\"older's inequality we get
\begin{eqnarray*}
&&\langle (-\Delta_p)^s u_n-(-\Delta_p)^s u, u_n-u\rangle\\
&=& \iint  \left(|\varphi_n(x,y)|^{p-2}\varphi_n(x,y)-|\psi(x,y)|^{p-2}\psi(x,y)\right)(\varphi_n(x,y)-\psi(x,y)) \\
&=&\iint   |\varphi_n(x,y)|^p + |\psi(x,y)|^{p}-|\varphi_n(x,y)|^{p-2}\varphi_n(x,y)\psi(x,y)\\
&&- \iint |\psi(x,y)|^{p-2}\varphi_n(x,y)\psi(x,y) \\
&\ge& [u_n]_{s,p}^p+[u]_{s,p}^p -[u_n]_{s,p}^{p-1}[u]_{s,p}-[u]_{s,p}^{p-1}[u_n]_{s,p}\\
&\ge & \left([u_n]_{s,p}^{p-1}-[u]_{s,p}^{p-1}\right)\left( [u_n]_{s,p}-[u]_{s,p}\right)\ge 0.
\end{eqnarray*}
The assumptions $u_n\rightharpoonup u$ and $\limsup_{n\to\infty}\langle (-\Delta_p)^su_n,u_n-u\rangle \le 0$ imply
$$
\limsup_{n\to\infty}\langle (-\Delta_p)^su_n-(-\Delta_p)^su,u_n-u\rangle \le 0,
$$
which together with the preceding inequality yields
$$
\lim_{n\to\infty}\left([u_n]_{s,p}^{p-1}-[u]_{s,p}^{p-1}\right)\left( [u_n]_{s,p}-[u]_{s,p}\right)=0,
$$
hence it follows $\lim_{n\to\infty}[u_n]_{s,p}=[u]_{s,p}$. Since $(D^{s,p}(\R^N), [\cdot]_{s,p})$ is a uniformly convex Banach space, the weak convergence $u_n\rightharpoonup u$ along with $ [u_n]_{s,p}\to [u]_{s,p}$ imply the  strong convergence, that is $[u_n-u]_{s,p}\to 0$ due to the Kadec-Klee property of  uniformly convex Banach spaces. This completes part (ii).
\end{proof}
\smallskip

\noindent Let $L^0(\mathbb{R}^N)$ denote the space of all real-valued measurable functions defined on $\mathbb{R}^N$, and Let $w: \mathbb{R}^N\to \mathbb{R}$ as in assumption (Ha), that is,
\begin{equation}\label{202}
w(x)=\frac{1}{1+|x|^{N+\alpha}},\quad\mbox{with }\alpha>0.
\end{equation}
By means of the function $w$ we introduce the following weighted Lebesgue space $L^q(\mathbb{R}^N, w)$ with $q\in (1, \infty)$ defined by %
$$
L^q(\mathbb{R}^N, w)=\Big\{u\in L^0(\mathbb{R}^N): \int_{\mathbb{R}^N} w|u|^q\,dx<\infty\Big\}
$$
which is a separable and reflexive Banach space under the norm
$$
\|u\|_{q,w}=\Big(\int_{\mathbb{R}^N} w|u|^q\,dx\Big)^{\frac{1}{q}}.
$$
Recall that  $\|\cdot\|_r$ stands for the  norm in $L^r(\mathbb{R}^N)$. We prove the following embedding result.
\begin{lemma}\label{L203}
The embedding $D^{s,p}(\R^N)\hookrightarrow  L^q(\mathbb{R}^N, w)$ is continuous for $1\le q\le p^*$, and $D^{s,p}(\R^N)\hookrightarrow\hookrightarrow L^q(\mathbb{R}^N, w)$
is compact for $1\le q<p^*$.
\end{lemma}
\begin{proof}  First we note that in view of Lemma \ref{201}, $w\in L^r(\mathbb{R}^N)$ for $1\le r\le \infty$.
Let $u\in D^{s,p}(\R^N)$, then $u\in L^{p^*}(\mathbb{R}^N)$, and thus for some positive constant we get for $1\le q< p^*$
$$
\int_{\mathbb{R}^N} w|u|^q\,dx\le \|w\|_{\frac{p^*}{p^*-q}}\|u\|_{p^*}^q\le c \|w\|_{\frac{p^*}{p^*-q}}[u]_{s,p}^q,
$$
that is,
$$
\|u\|_{q,w}\le c \|w\|_{\frac{p^*}{p^*-q}}^{\frac{1}{q}}[u]_{s,p},
$$
which shows that $i_w: D^{s,p}(\R^N)\to L^q(\mathbb{R}^N, w)$ is linear and continuous for $1\le q< p^*$. If $q=p^*$, then due to $w\in L^\infty(\R^N)$,  the inequality
$$
\|u\|_{p^*,w}\le c  [u]_{s,p}
$$
is trivially satisfied, which proves the first part.

As for the compact embedding for $1\le q < p^*$, let $(u_n)\subset D^{s,p}(\R^N)$ be bounded. Since $D^{s,p}(\R^N)$ is reflexive, there exists a weakly convergent subsequence still denoted by $(u_n)$, that is, $u_n\rightharpoonup u$ in $D^{s,p}(\R^N)$, which due to $D^{s,p}(\R^N)\hookrightarrow L^{p^*}(\R^N)$ yields $u_n\rightharpoonup u$ in $L^{p^*}(\R^N)$.  On the other hand $(u_n)\subset W^{s,p}(B_R)$ is bounded for any $R$, where   $B_R=B(0,R)$ is the open ball with radius $R$, and $W^{s,p}(B_R)$ is the fractional Sobolev space defined by
$$
W^{s,p}(B_R)=\left\{u\in L^p(B_R): \iint_{B_R\times B_R}\frac{|u(x)-u(y)|^p}{|x-y|^{N+sp}}\,dxdy<\infty\right\},
$$
endowed with the norm
$$
\|u\|_{W^{s,p}(B_R)}=\left(\|u\|^p_{L^p(B_R)}+[u]_{s,p,B_R}^p\right)^{\frac{1}{p}},
$$
where
$$
[u]_{s,p,B_R}^p=\iint_{B_R\times B_R}\frac{|u(x)-u(y)|^p}{|x-y|^{N+sp}}\,dxdy.
$$
The space $W^{s,p}(B_R)$ is a uniformly convex Banach space and thus reflexive. Also we have  $L^{p^*}(\R^N)\hookrightarrow L^{p^*} (B_R)\hookrightarrow L^{q}(B_R)$, hence  $u_n\rightharpoonup u$ in $L^{q}(B_R)$. Since $(u_n)$ is bounded in $W^{s,p}(B_R)$ and $W^{s,p}(B_R)\hookrightarrow\hookrightarrow L^q(B_R)$ is compactly embedded for $1\le q< p^*$ (see, e.g., \cite[Theorem 4.54]{Demengel-2012}), there is a subsequence again denoted by $(u_n)$ such that $u_n\to v$ in $L^q(B_R)$ which by passing again to a subsequence yields $u_n(x)\to v(x)$ for a.a. $x\in B_R$. The weak convergence  $u_n\rightharpoonup u$ in $L^{q}(B_R)$ along with  $u_n(x)\to v(x)$ for a.a. $x\in B_R$ implies that $u=v$ in $B_R$. Let us  show that in fact $u_n\to u$ in $L^q(\R^N, w)$.

For $\varepsilon>0$  arbitrarily be given and any $R>0$, we consider
\begin{equation}\label{203}
\|u_n-u\|_{q,w}^q=\int_{\mathbb{R}^N\setminus B_R} w|u_n-u|^q\,dx+\int_{B_R} w|u_n-u|^q\,dx.
\end{equation}
Since $(u_n)$ is bounded in $D^{s,p}(\R^N)$, that is $[u_n]_{s,p}\le c$,   it is also  bounded in $L^{p^*}(\mathbb{R}^N)$, and thus with some generic constant $c$ independent of $n$ and $R$,  we can estimate the first integral on the right-hand side of (\ref{203}) as follows:
\begin{eqnarray*}
&&\int_{\mathbb{R}^N\setminus B_R} w|u_n-u|^q\,dx \le  c\int_{\mathbb{R}^N\setminus B_R} w\big(|u_n|^q+|u|^q\big)\,dx\\
&&\le  c \|w\|_{L^{\frac{p^*}{p^*-q}}(\mathbb{R}^N\setminus B_R)}\Big(\|u_n\|^q_{L^{p^*}(\mathbb{R}^N\setminus B_R)}+\|u\|^q_{L^{p^*}(\mathbb{R}^N\setminus B_R)}\Big)\\
&&\le  c \|w\|_{L^{\frac{p^*}{p^*-q}}(\mathbb{R}^N\setminus B_R)}\Big(\|u_n\|^q_{p^*}+\|u\|^q_{p^*}\Big),
\end{eqnarray*}
which, together with the estimate $\|u_n\|_{p^*}\le c[u_n]_{s,p}\le c$, yields
\begin{equation}\label{204}
\int_{\mathbb{R}^N\setminus B_R} w|u_n-u|^q\,dx\le c\,\|w\|_{L^{\frac{p^*}{p^*-q}}(\mathbb{R}^N\setminus B_R)}.
\end{equation}
The right-hand side of (\ref{204}) can be further estimated as
\begin{equation}\label{205}
\|w\|_{L^{\frac{p^*}{p^*-q}}(\mathbb{R}^N\setminus B(0,R))}^{\frac{p^*}{p^*-q}}\le c\int_R^\infty\Big(\frac{1}{1+\varrho^{N+\alpha}}\Big)^{\frac{p^*}{p^*-q}}\varrho^{N-1}\,d\varrho\le c R^{-(N+\alpha){\frac{p^*}{p^*-q}}+N},
\end{equation}
since $-(N+\alpha){\frac{p^*}{p^*-q}}+N<0$. It follows from (\ref{204}) and (\ref{205}) the existence of $R>0$ sufficiently large such that
\begin{equation}\label{206}
\int_{\mathbb{R}^N\setminus B_R} w|u_n-u|^q\,dx< \frac{\varepsilon}{2}, ~~\forall\  n\in \mathbb{N}.
\end{equation}
In view of $u_n\to u$ (strongly) in $ L^q(B_R)$ and taking into account that $w\in L^\infty(\mathbb{R}^N)$ one gets
\begin{equation}\label{207}
\int_{B(0,R)} w|u_n-u|^q\,dx< \frac{\varepsilon}{2} \ \mbox{ for $n$ sufficiently large}.
\end{equation}
The estimates (\ref{206}) and (\ref{207}) complete the proof.
\end{proof}
%
%%%%%%%%%%%%%SUBSECTION32%%%%%%%%%%%%%%%%%%%%%%%%%%%%%%%%%%%
\subsection{Decay estimate via Wolff potential}\label{S32}
%%%%%%%%%%%%%%%%%%%%%%%%%%%%%%%%%%%%%%%%%%%%%%%%%%%%%%%%%%%
%%%%%%%%%%%%%%%%%%%%%%%%%%%%%%%%%%%%%%%%%%%%%%%%%%%%%%%%%%%%
In this section we are going to prove a decay estimate for the unique solution of the fractional $p$-Laplacian equation
\begin{equation}\label{401}
(-\Delta _p)^s u(x)=|a(x)| \quad\mbox{in }\R^N.
\end{equation}
Assuming (Ha), then $a$ satisfies (A1), and thus,  in particular, $a\in L^{{p*}'}(\R^N)\subset (D^{s,p}(\R^N))^*$. Hence by the continuity and strict monotonicity of the fractional $p$-Laplacian, we get the existence of a unique solution $u\in D^{s,p}(\R^N)$ of (\ref{401}). From \cite{MS-2016} we deduce that the solution is loc-H\"older, that is, it belongs to $C_{\mathrm{loc}}^{0,\gamma}(\R^N)$, and thus, in particular continuous. Moreover,  $u\ge 0$, and if $u \not\equiv 0$ then $u(x)>0$ in $\R^N$ and $\inf_{\R^N}u(x)=0$, see, e.g., \cite{DKP-2014}.

The goal of this subsection is to prove the following pointwise upper bound for the solution of (\ref{401}).
\begin{theorem}\label{T401}
Assume hypothesis (Ha). The positive, continuous solution $u$ of the fractional $p$-Laplacian equation (\ref{401}) has the following  pointwise upper bound:
\begin{equation}\label{402}
u(x)\le C |x|^{-\frac{N-sp}{p-1}}, \ \forall\ x: |x|\ge 1,
\end{equation}
where $C=C(N,p,s)$.
\end{theorem}
Our proof will be based on a recent result on pointwise upper bounds for superharmonic functions  of a nonlocal operator $\mathcal{L}$ obtained in \cite[p.3, formula (1.3)]{KLL-2025}, a special case of which is the fractional $p$-Laplacian when specifying the function $g$ and $k=k(x,y)$ appearing in the definition of $\mathcal{L}$ as $g(t)=t^{p-1}$ and $k(x,y)\equiv 1$, respectively. Let us recall the general result on pointwise upper bound given in \cite{KLL-2025}.
\begin{theorem}\cite[Theorem 1.2]{KLL-2025}\label{T402}
Let $s\in (0,1)$. Let $u$ be a
a superharmonic function in $B_R(x_0)$ such that $u \ge 0$ in $B_R(x_0)$. Then $\mu :=\mathcal{L}u$ is a nonnegative Borel measure and
\begin{equation}\label{403}
u(x_0)\le C\left(\inf_{B_{R/2}(x_0)}u+W^\mu_{s,G}(x_0,R)+\mathrm{Tail}_g(u;x_0,R/2)\right)
\end{equation}
for some $C=C(N, p, q, s, \Lambda) >0$.
\end{theorem}
Applying this general result about $\mathcal{L}$-superharmonic functions to $(-\Delta_p)^s$-superharmonic functions, the specific nonnegative and finite Borel measure $\mu$ is defined by the nonnegative function $|a|\in L^1(\mathbb{R}^N)\cap L^\infty(\mathbb{R}^N)$  through
\begin{equation}\label{404}
\mu(E):=\int_{E}|a(x)|\,dx, \ \ E\subset\mathbb{R}^N, \ \ E\mbox{ Lebesgue measurable}.
\end{equation}
The Wolff potential $W^\mu_{s,G}(x_0,R)$ in  (\ref{403}) reduces to the classical nonlinear Wolff potential $W^{\mu}_{s,p}(x,R)$  given by
\begin{equation}\label{405}
W^{\mu}_{s,p}(x,R)=\int_0^R\Big(\frac{|\mu|(B_t(x))}{t^{N-sp}}\Big)^{\frac{1}{p-1}}\frac{dt}{t},\ \ R>0,\ x\in \mathbb{R}^N, \ s\in \Big(0,\frac{N}{p}\Big),
\end{equation}
and the tail function $\mathrm{Tail}_g(u;x_0,R)$ appearing  in (\ref{403}) reduces in the special case of $(-\Delta_p)^s$-superharmonic functions to
\begin{equation}\label{406}
\mathrm{Tail}(u;x_0,R)=\left(R^{sp}\int_{\R^N\setminus B_R(x_0)}\frac{|u(x)|^{p-1}}{|x-x_0|^{N+sp}}\,dx\right)^{\frac{1}{p-1}}.
\end{equation}
The role of the tail is to control the growth of solutions at infinity.

Since the continuous solution of (\ref{401}) is a continuous supersolution of the equation $(-\Delta_p)^s u=0$, which is a superharmonic function for the fractional $p$-Laplacian (see, e.g., \cite[Theorem 2.10]{KLL-2025}), we deduce the following pointwise upper bound result from Theorem \ref{T402}.
\begin{theorem}\label{T403}
Assume hypothesis (Ha). The positive and continuous solution of (\ref{401}) has the following pointwise upper bound
\begin{equation}\label{407}
u(x_0)\le C\left(\inf_{B_{R/2}(x_0)}u+W^{\mu}_{s,p}(x_0,R)+\mathrm{Tail}(u;x_0,R/2)\right)
\end{equation}
for some $C=C(N, p,s) >0$, where $W^{\mu}_{s,p}(x_0,R)$ and $\mathrm{Tail}(u;x_0,R)$ are given by (\ref{405}) and (\ref{406}), respectively.
\end{theorem}
Since the upper bound of (\ref{407}) holds true for any $R>0$, in what follows we are going to provide estimates for the terms on the right-hand side as $R\to \infty$. Observing that $\inf_{\R^N}u=0$, from (\ref{407}) we get
\begin{equation}\label{408}
u(x_0)\le C\left( W^{\mu}_{s,p}(x_0,\infty)+\mathrm{Tail}(u;x_0,\infty)\right),
\end{equation}
where $C=C(N,p,s)$ is some positive constant.

In what follows we are going to provide estimates of the terms on the right-hand side of (\ref{408}). To this end we prove the following auxilliary result. We use either $B(x,t)$ or $B_t(x)$ when suitable for the ball with center $x$ and radius $t$.
\begin{lemma}\label{L401}
Assume hypothesis (Ha). Then the following estimate holds true:
\begin{equation}\label{409}
\int_{B(x,t)}|a(y)|\,dy\le C\,t^{\frac{N}{\sigma'}}\|a\|_{L^\sigma(B(x,t))}
\end{equation}
\end{lemma}
\begin{proof}
By Lemma \ref{L201} $x\to |a(x)|$ has the properties (A1)--(A2), thus we can estimate  $\int_{B(x,t)}|a(y)|\,dy$ as follows:
\begin{eqnarray*}
\int_{B(x,t)}|a(y)|\,dy&=& \int_{B(0,1)}|a(x+tz)|t^N\,dz\le C\,t^N\bigg(\int_{B(0,1)}|a(x+tz)|^\sigma\,dz\bigg)^\frac{1}{\sigma}\\
&\le& C t^N t^{-\frac{N}{\sigma}} \bigg(\int_{B(x,t)}|a(y)|^\sigma\,dy\bigg)^\frac{1}{\sigma}\\
&\le& C t^{N-\frac{N}{\sigma}}\|a\|_{L^\sigma(B(x,t))}=C\,t^{\frac{N}{\sigma'}}\|a\|_{L^\sigma(B(x,t))},
\end{eqnarray*}
which completes the proof.
\end{proof}
Let us estimate the Wolff potential $W^{\mu}_{s,p}(x_0,\infty)$.
\begin{eqnarray*}
W^{\mu}_{s,p}(x;\infty)&=&\int_0^\infty\Big(\frac{|\mu|(B(x,t))}{t^{N-sp}}\Big)^{\frac{1}{p-1}}\frac{dt}{t}\\
&=&\int_0^\infty\bigg(\int_{B(x,t)} |a(y)|\,dy\bigg)^{\frac{1}{p-1}}t^{\frac{sp-N}{p-1}-1}\,dt\\
&=&\int_0^{\frac{|x|}{2}}\bigg(\int_{B(x,t)} |a(y)|\,dy\bigg)^{\frac{1}{p-1}}t^{\frac{sp-N}{p-1}-1}\,dt\\
&& +\int_{\frac{|x|}{2}}^\infty\bigg(\int_{B(x,t)} |a(y)|\,dy\bigg)^{\frac{1}{p-1}}t^{\frac{sp-N}{p-1}-1}\,dt\\
&=& I_1+I_2.
\end{eqnarray*}
We estimate $I_1$ and $I_2$ separately.
By (A1) we have $a\in L^1(\R^N)$, which yields for $I_2$ (note: $sp<N$)
$$
I_2\le \|a\|_1^\frac{1}{p-1}\int_{\frac{|x|}{2}}^{\infty}t^{\frac{sp-N}{p-1}-1}\,dt\le
C\|a\|_1^\frac{1}{p-1} \Big({\frac{|x|}{2}}\Big)^{-{\frac{N-sp}{p-1}}},
$$
which results in
\begin{equation}\label{410}
I_2\le C\, |x|^{-{\frac{N-sp}{p-1}}},
\end{equation}
where $C=C(a,N,p,s)$ is some positive constant. As for the estimate of $I_1$ we make use of Lemma \ref{L401}, and
$$
\|a\|_{L^\sigma(B(x,t))}\le \|a\|_{L^\sigma\big(\mathbb{R}^N\setminus B(0,\frac{|x|}{2})\big)},\ \mbox{ for all } t\in \Big[0, \frac{|x|}{2}\Big],
$$
as well as of $\sigma>\frac{N}{sp}$.
\begin{eqnarray*}
I_1 &= &\int_0^{\frac{|x|}{2}}\bigg(\int_{B(x,t)} |a(y)|\,dy\bigg)^{\frac{1}{p-1}}t^{\frac{sp-N}{p-1}-1}\,dt\\
&\le& C\int_0^{\frac{|x|}{2}}\|a\|_{L^\sigma(B(x,t))}^{\frac{1}{p-1}}t^{\frac{N}{\sigma'}\frac{1}{p-1}+\frac{sp-N}{p-1}-1}\,dt\\
&\le & C\,\|a\|_{L^\sigma (\mathbb{R}^N\setminus B(0,\frac{|x|}{2}) )}^{\frac{1}{p-1}}
\int_0^{\frac{|x|}{2}}t^{\frac{N}{\sigma'}\frac{1}{p-1}+\frac{sp-N}{p-1}-1}\,dt\\
&\le & C\,\|a\|_{L^\sigma (\mathbb{R}^N\setminus B(0,\frac{|x|}{2}) )}^{\frac{1}{p-1}}
\Big(\frac{|x|}{2}\Big)^{\frac{N}{\sigma'}\frac{1}{p-1}+\frac{sp-N}{p-1}} \\
&\le & C\,\Big[\Big(\frac{|x|}{2}\Big)^{\frac{N}{\sigma'}}\|a\|_{L^\sigma (\mathbb{R}^N\setminus B(0,\frac{|x|}{2}))} \Big]^{\frac{1}{p-1}}\Big(\frac{|x|}{2}\Big)^{\frac{sp-N}{p-1}},
\end{eqnarray*}
which by using  (A2) yields for some positive constant $C=C(a,N,p,S)>0$
\begin{equation}\label{411}
I_1\le C\, |x|^{-\frac{N-sp}{p-1}},\ \ |x|>0.
\end{equation}
By (\ref{410}) and (\ref{411}) we obtain the estimate for the Wolff potential
\begin{equation}\label{412}
W^{\mu}_{s,p}(x;\infty)=I_1+I_2\le C\, |x|^{-\frac{N-sp}{p-1}},\ \ |x|>0.
\end{equation}

Next let us estimate the tail function $\mathrm{Tail}(u;x_0,\infty)$, that is, $\lim_{R\to\infty}\mathrm{Tail}(u;x_0,R)$, with $\mathrm{Tail}(u;x_0,R)$ given by (\ref{406}). By applying H\"older inequality we get
\begin{eqnarray*}
&&\int_{\R^N\setminus B_R(x_0)}\frac{|u(x)|^{p-1}}{|x-x_0|^{N+sp}}\,dx \le \left(\int_{\R^N\setminus B_R(x_0)}|u(x)|^{p^*}\,dx\right)^{\frac{p-1}{p^*}}\\
&&\qquad \times\left(\int_{\R^N\setminus B_R(x_0)}\left(\frac{1}{|x-x_0|^{N+sp}}\right)^{\frac{p^*}{p^*-p+1}}\,dx\right)^{\frac{p^*-p+1}{p^*}}\\
&&\le C [u]_{s,p}^{p-1}\left(\int_{\R^N\setminus B_R(x_0)}\left(\frac{1}{|x-x_0|^{N+sp}}\right)^{\frac{p^*}{p^*-p+1}}\,dx\right)^{\frac{p^*-p+1}{p^*}}.
\end{eqnarray*}
Estimating the second term on the right-hand side of the last inequality using spherical coordinates we get
\begin{eqnarray*}
&&\int_{\R^N\setminus B_R(x_0)}\left(\frac{1}{|x-x_0|^{N+sp}}\right)^{\frac{p^*}{p^*-p+1}}\,dx\le C\int_R^\infty \frac{t^{N-1}}{t^{(N+sp ) \frac{p^*}{p^*-p+1}}}\,dt\\
&&\le C R^{ N-(N+sp ) \frac{p^*}{p^*-p+1}},
\end{eqnarray*}
since $ N-(N+sp ) \frac{p^*}{p^*-p+1}<0$. The last two inequalities yield
\begin{equation}\label{413}
\int_{\R^N\setminus B_R(x_0)}\frac{|u(x)|^{p-1}}{|x-x_0|^{N+sp}}\,dx\le C [u]_{s,p}^{p-1}R^{ N\frac{p^*-p+1}{p^*}-(N+sp) },
\end{equation}
where $C=C(N,p,s)>0$. Plugging (\ref{413}) into (\ref{406}) results in
\begin{eqnarray*}
\mathrm{Tail}(u;x_0,R)&=&\left(R^{sp}\int_{\R^N\setminus B_R(x_0)}\frac{|u(x)|^{p-1}}{|x-x_0|^{N+sp}}\,dx\right)^{\frac{1}{p-1}}\\
&\le& C [u]_{s,p} R^{\frac{sp}{p-1}}R^{{ N\frac{p^*-p+1}{p^*(p-1)}-(N+sp) }\frac{1}{p-1}}\\
&\le &  C [u]_{s,p} R^\frac{-N}{p^*},
\end{eqnarray*}
where $C=C(N,p,s)$ is some positive constant not depending on $R$. Hence, the last inequality results in
\begin{equation}\label{414}
\mathrm{Tail}(u;x_0,\infty)=\lim_{R\to\infty}\mathrm{Tail}(u;x_0,R)=0, \ \mbox{ for any } x_0\in \R^N.
\end{equation}

\smallskip

\noindent{\bf Proof of Theorem \ref{T401}}  The proof of the pointwise upper bound (\ref{402}) is an immediate consequence of the estimates (\ref{408}), (\ref{412}), and (\ref{414}).\hfill $\Box$

\smallskip

\begin{corollary}\label{C401}
Under the assumptions of Theorem \ref{T401} we have the following pointwise upper bound in the whole $\R^N$:
\begin{equation}\label{415}
u(x)\le \tilde C \frac{1}{1+|x|^\frac{N-sp}{p-1}}, \quad x\in \R^N,
\end{equation}
where $\tilde C=\tilde C(N,p,s)>0$.
\end{corollary}
\begin{proof}
Since $u(x)>0$ and continuous in $\R^N$, we have  $\delta=\max_{|x|\le 1} u(x)>0$, which implies
\begin{equation}\label{416}
u(x)\le \frac{2\, \delta}{1+|x|^{\frac{N-sp}{p-1}}},\ \ |x|\le 1,
\end{equation}
and for $|x|\ge 1$, by Theorem  \ref{T401} we obtain
\begin{equation}\label{417}
u(x)\le 2C\,\frac{1}{1+|x|^{\frac{N-sp}{p-1}}},\ \ |x|\ge 1.
\end{equation}
Hence  (\ref{416})--(\ref{417}) with $\tilde C= 2\max\{\delta, C\}$ proves (\ref{415}).
\end{proof}
%
%%%%%%%%%%%%%SECTION4%%%%%%%%%%%%%%%%%%%%%%%%%%%%%%%%%%%
\section{Proof of Theorem \ref{D-T101}}\label{S4}
%%%%%%%%%%%%%%%%%%%%%%%%%%%%%%%%%%%%%%%%%%%%%%%%%%%%%%%%%%%
%%%%%%%%%%%%%%%%%%%%%%%%%%%%%%%%%%%%%%%%%%%%%%%%%%%%%%%%%%%%
Assume (Hg)--(Ha), and let $u\in D^{s,p}(\R^N)$ be a solution of (\ref{101}) (resp. (\ref{G-401}),  that is,
$$
(-\Delta_p)^su=g(x,u)\quad\mbox{in }\R^N.
$$
Applying  Corollary \ref{G-C401}  the right-hand side of (\ref{101}) can be estimated by
\begin{equation}\label{G-406}
|g(x,u(x)|\le C |a(x)|=: \hat{a}(x),\quad\forall x\in \R^N,
\end{equation}
where $C=C(p,q,N,s, \|a\|, \|u\|_{p^*},[u]_{s,p})$. Thus, for a given solution $u$, $\hat{a}$ satisfies (Ha). Let $v$ be the unique solution of
\begin{equation}\label{G-407}
v\in D^{s,p}(\R^N): (-\Delta_p)^sv=\hat{a}(x)\quad\mbox{in }\R^N,
\end{equation}
and let $z$ be the unique solution of
\begin{equation}\label{G-408}
z\in D^{s,p}(\R^N): (-\Delta_p)^sz=-\hat{a}(x)\quad\mbox{in }\R^N.
\end{equation}
The comparison principle proved in \cite[Lemma 9]{LL-2014} can easily be extended to the fractional $p$-Laplacian equation considered here, which results in $z\le 0\le v$. Further, from (\ref{G-406}) and (\ref{G-407}) we get (in the weak sense)
$$
\langle (-\Delta_p)^sv-(-\Delta_p)^su, \varphi\rangle \ge 0,\quad\forall \varphi\in D^{s,p}(\R^N), \ \varphi\ge 0,
$$
that is,
\begin{eqnarray*}
&&\int_{\R^N}\int_{\R^N}\frac{|v(x)-v(y)|^{p-2}(v(x)-v(y))(\varphi(x)-\varphi(y))}{|x-y|^{N+sp}} dxdy\\
&& \ge \int_{\R^N}\int_{\R^N}\frac{|u(x)-u(y)|^{p-2}(u(x)-u(y))(\varphi(x)-\varphi(y))}{|x-y|^{N+sp}} dxdy,
\end{eqnarray*}
which by taking the special test function $\varphi=(u-v)^+:=\max\{u-v,0\}$ and applying the comparison principle yields $u\le v$. Multiplying (\ref{G-401}) and (\ref{G-408}) by $-1$  and setting $\hat z=-z$ and $\hat u=-u$ results in
\begin{equation}\label{G-409}
(-\Delta_p)^s\hat{u}=-g(x,-\hat{u}), \quad  (-\Delta_p)^s\hat z=\hat{a}(x)\quad\mbox{in }\R^N,
\end{equation}
and thus again by comparison $\hat u\le \hat z$, that is, $u\ge z$. Now we may apply the decay result of Theorem \ref{T401} to the nonnegative solutions $v$ and $\hat z$, which yields
\begin{equation}\label{G-410}
v(x)\le C |x|^{-\frac{N-sp}{p-1}}, \quad \hat z(x)\le C |x|^{-\frac{N-sp}{p-1}},\ \forall\ x: |x|\ge 1,
\end{equation}
where $C=C(p,q,N,s, \|a\|, [u]_{s,p})$. Thus we get
$$
-C |x|^{-\frac{N-sp}{p-1}}\le u(x)\le C |x|^{-\frac{N-sp}{p-1}}, \ \forall\ x: |x|\ge 1,
$$
which completes the proof of Theorem \ref{D-T101}. \hfill $\Box$

An immediate consequence of  Corollary \ref{C401} and Theorem \ref{D-T101} is the following decay estimate in all $\R^N$ for a solution $u$ of (\ref{101}):
\begin{corollary}\label{R4-C401}
$$
|u(x)|\le C \frac{1}{1+|x|^\frac{N-sp}{p-1}}, \quad x\in \R^N,
$$
where $C$ is as in Theorem \ref{D-T101}.

\end{corollary}
%

%
%%%%%%%%%%%%%SECTION5%%%%%%%%%%%%%%%%%%%%%%%%%%%%%%%%%%%
\section{$D^{s,2}(\R^N)$ versus $C_b\left(\R^N, 1+|x|^{N-2s}\right)$ local minimizers}\label{S5}
%%%%%%%%%%%%%%%%%%%%%%%%%%%%%%%%%%%%%%%%%%%%%%%%%%%%%%%%%%%
%%%%%%%%%%%%%%%%%%%%%%%%%%%%%%%%%%%%%%%%%%%%%%%%%%%%%%%%%%%%
Let $V_s=D^{s,2}(\R^N)\cap C_b\left(\R^N, 1+|x|^{N-2s}\right)$ be the subspace of  bounded continuous functions with weight $1+|x|^{N-2s}$ defined by
$$
V_s := \left\{\, v\in D^{s,2}(\R^N): v\in C(\R^N)\, \ \mbox{with}\,\ \sup_{x\in\R^N}\big(1 +|x|^{N-2s}\big)|v(x)| < \infty\, \right\},
$$
which is a closed subspace of $D^{s,2}(\R^N)$ with the norm
$$
\|v\|_{V_s} := [v]_{s,2}+ \sup_{x\in\R^N}\big(1 +|x|^{N-2s}\big)|v(x)|.
$$
Consider the fractional Laplacian equation
\begin{equation}\label{R4:201}
u\in D^{s,2}(\R^N): (-\Delta)^s u=a(x)g(u),
\end{equation}
where $a: \R^N\to \R$ satisfies (Ha) and $g: \R\to\R$ is continuous and has the growth
\begin{equation}\label{R4:202}
|g(s)|\le c_g\left(1+|s|^{\gamma-1}\right),\quad 1\le \gamma< 2^*=\frac{2N}{N-2s}.
\end{equation}
The energy functional $\Phi: D^{s,2}(\R^N)\to \R$ related to the fractional elliptic equation (\ref{R4:201}) is given by
\begin{equation}\label{R4:203}
\Phi(u)= \frac{1}{2} [u]_{s,2}^2-\int_{\R^N}a(x) G(u)dx, \quad\mbox{with } G(t)=\int_0^tg(\tau) d\tau.
\end{equation}
One readily verifies that $\Phi$ is a $C^1$ functional and its Frechet derivative at $u\in D^{s,2}(\R^N)$ is given by
\begin{eqnarray} \label{R4:204}
\langle \Phi'(u),\varphi\rangle &=&\iint_{\R^N\times\R^N}\frac{|u(x)-u(y)|(u(x)-u(y)(\varphi(x)-\varphi(y)}{|x-y|^{N+2s}}\,dxdy\nonumber \\
&& -\int_{\R^N}a(x)g(u)\varphi(x)\,dx,\quad \forall\ u,\varphi\in D^{s,2}(\R^N).
\end{eqnarray}
\begin{lemma}\label{R4: L401}
Let hyotheses (Ha) and (\ref{R4:202}) be satisfied. Then the functional $\Phi: D^{s,2}(\R^N)\to \R$ is $C^1$, weakly lower semicontinous, and critical points of $\Phi$ are weak solutions of the fractional elliptic equation (\ref{R4:201}).
\end{lemma}
\begin{proof}
We only need to show that $\Phi$ is weakly lower semicontinous. The norm $[\cdot]_{s,2}$  is weakly lower semicontinuous. From Lemma \ref{L203}  it follows that $D^{s,2}(\R^N)\hookrightarrow\hookrightarrow L^{\gamma}(\R^N,w)$ for $1\le \gamma < 2^*$, and thus
the functional $u\mapsto \int_{\R^N}a(x) G(u)dx$ is weakly lower semicontinuous, which completes the proof.
\end{proof}
The main result of this section reads as follows.
\begin{theorem}\label{T501}
Let $a: \R^N\to \R $ fulfill (Ha), and let the continuous function $g: \R\to\R$ satisfy (\ref{R4:202}).
Suppose $u_0\in D^{s,2}(\R^N)$ is  a solution of the equation (\ref{R4:201}) and a local minimizer in the $V_s$-topology of the  functional $\Phi: D^{s,2}(\R^N)\to  \mathbb{R}$, that is,
there exists  $\varepsilon>0$ such that
$$
\Phi(u_0)\le \Phi(u_0+h), \ \ \forall\ h\in V_s: \|h\|_{V_s}<\varepsilon.
$$
Then $u_0$ is a local minimizer of $\Phi$ with respect to the $[\cdot]_{s,2}$-topology of $D^{s,2}(\R^N)$,  that is, there is  $\delta>0$ such that
$$
\Phi(u_0)\le \Phi(u_0+h), \ \ \forall\ h\in D^{s,2}(\R^N): [h]_{s,2}<\delta.
$$
\end{theorem}
Theorem \ref{T501} is in the spirit of a classical result due to Brezis and Nirenberg for a semilinear elliptic Dirichlet problem  on bounded domains $\Omega$ of the form
\begin{equation}\label{500}
u\in W^{1,2}_0(\Omega): -\Delta u=f(x,u) \mbox{ in }\Omega,\quad u=0 \mbox{ on }\partial\Omega,
\end{equation}
where $f$ is a Carath\'eodory function with $u\mapsto f(\cdot,u)$ satisfying a subcritical growth, and $W^{1,2}_0(\Omega)$ denotes the usual Sobolev space. In \cite{BN}, for the associated energy functional
$$
u\mapsto \frac{1}{2}\int_{\Omega}|\nabla u|^2\,dx-\int_{\Omega}F(x,u)\, dx, \quad\mbox{with }F(\cdot,t)=\int_0^t f(\cdot,\tau)\,d\tau,
$$
a $W^{1,2}_0(\Omega)$ versus $C^1(\Omega)$ local minimizers theorem is proved. Extensions of the  Brezis-Nirenberg result on bounded domains with leading $p$-Laplacian type variational operators have been obtained by several authors (see, e.g., \cite{BGWZ20, GPM00, GST07, Guo-Zhang03, IMS-2020, IMS15,ST13} ). The literature about extensions to unbounded domains, in particular to the whole $\R^N$, is  much less developed.  Extensions to $\R^N$ with  the  Laplacian  as leading operator within the Beppo-Levi space $D^{1,2}(\R^N)$   can be found in  \cite{CCT-JDE, CCT-ACV}, and with $p$-Laplacian equations within $D^{1,p}(\R^N)$ we refer to \cite{CT-2025}. An extension of the Brezis-Nirenberg result to the (unbounded) exterior domain $\R^N\setminus \overline{B(0,1)}$ was obtained in \cite{Carl-20} for the $N$-Laplacian equation in the Beppo-Levi space $D^{1,N}_0(\R^N\setminus \overline{B(0,1)}$,  which is  based on Kelvin transform. The latter, however, only works for $p$-Laplacian equations with $p=2$ or $p=N$.
Finally in the recent paper \cite{Ambrosio-23}, partially motivated by the work in \cite{CCT-JDE}, the author proves a result similar to Theorem \ref{T501} above under the following strong hypotheses on the nonlinearity $g$, see \cite[(G1),(G2)]{Ambrosio-23},
\begin{itemize}
\item[(G1)] $g\in C(\R)$ and $\lim_{|t|\to 0}\frac{g(t)}{|t|}=0 $;
\item[(G2)] there exists $ C > 0$   and $p\in (2,2^*)$ such that
$$
|g(t_1)-g(t_2)|\le C|t_1-t_2|[1+\max\{|t_1|,|t_2|\}^{p-2}
$$
for all $t_1,\,t_2\in\R$.
\end{itemize}
Here, taking advantage of the $L^{\infty}$ and decay estimates proved in the previous two sections for general fractional $p$-Laplacian operators, specializing to the case $p=2$, we are able to prove this result only assuming the weak growth condition
 (\ref{R4:202}) on $g$, improving the result in \cite{Ambrosio-23}.

\smallskip

\centerline{\bf Proof  Theorem \ref{T501}}

\smallskip

\noindent

Let $u_0$ be a solution of (\ref{R4:201}) and a local minimizer of the functional $\Phi$ given by (\ref{R4:203})
in the $V_s$-topology. First we note that by Corollary \ref{G-C401}, the solution $u_0\in L^{\infty}(\R^N)$.
Consider the functional $h\mapsto \Phi(u_0+h)$, and let $h_n: [h_n]_{s,2}\le \frac{1}{n}$  be such that
$$
\Phi(u_0+h_n)=\inf_{h\in B_n} \Phi(u_0+h),\quad\mbox{where } B_n=\Big\{h\in D^{s,2}(\R^N): [h]_{s,2}\le \frac{1}{n}\Big\}.
$$
The existence of a minimizer $h_n$ is guaranteed, since $\Phi: D^{s,2}(\R^N)\to \R$ is $C^1$ and weakly lower semicontinuous and $B_n$ is weakly compact in $D^{s,2}(\R^N)$. Set $u_n=u_0+h_n$, that is,
$$
\Phi(u_n)=\inf_{u\in B_n} \Phi(u),\quad\mbox{where } B_n=\Big\{u\in D^{s,p}(\R^N): [u-u_0]_{s,2}\le \frac{1}{n}\Big\}.
$$
For $u_n\in B_n$ we have either $[u_n-u_0]_{s,2}<\frac{1}{n}$ or else $[u_n-u_0]_{s,2}=\frac{1}{n}$. In case $[u_n-u_0]_{s,2}<\frac{1}{n}$, $u_n$ is a critical point of $\Phi$, and thus $u_n$ is a weak solution of (\ref{R4:201}), i.e., $(-\Delta)^s u_n=a(x) g(u_n)$.
In case $[u_n-u_0]_{s,2}=\frac{1}{n}$, by Lagrange multiplier rule, there exists a Lagrange multiplier $\lambda_n\le 0$ such that
(in weak sense)
\begin{equation}\label{501}
(-\Delta)^s u_n-a(x) g(u_n)=\lambda_n(-\Delta)^s(u_n-u_0).
\end{equation}
Taking into account that $u_0$ is a solution of (\ref{R4:201}) and using (\ref{501}), we get
\begin{equation}\label{502}
(-\Delta)^s(u_n-u_0)-\lambda_n(-\Delta)^s(u_n-u_0)= a(x) (g(u_n)-g(u_0)),
\end{equation}
and thus  $h_n=u_n-u_0$ satisfies the equation
\begin{equation}\label{503}
(-\Delta)^sh_n= \frac{1}{1+\mu_n}a(x) (g(u_0+h_n)-g(u_0)),
\end{equation}
where $\mu_n:=-\lambda_n\ge 0$. Set $f_n(x,t)=\frac{1}{1+\mu_n}a(x) (g(u_0(x)+t)-g(u_0(x)))$, then
\begin{equation}\label{504}
|f_n(x,t)|\le C|a(x)|\left(1+|t|^{\gamma-1}\right), \quad\mbox{uniformly }\forall n,
\end{equation}
where $C=C(c_g,\gamma,\|u_0\|_{\infty})$. Set $\hat{a}(x)=C|a(x)|$, then clearly $\hat{a}$ satisfies (Ha), which along with the growth estimate (\ref{504}) allows us to apply Corollary \ref{G-C401} to the solutions $h_n$ of (\ref{503}), and obtain
\begin{equation}\label{505}
\|h_n\|_{\infty}\le C\max\left\{\|h_n\|_{2^*}^{\theta_0}, \|h_n\|_{2^*}\right\}\le C\max\left\{[h_n]_{s,2}^{\theta_0}, [h_n]_{s,2}\right\}
\end{equation}
where $C=C(N,s,\|\hat a\|, [h_n]_{s,2})$. Since $[h_n]_{s,2}\to 0$ as $n\to \infty$, from (\ref{505}) it follows that $\|h_n\|_{\infty}\to 0$, which in view of the uniform continuity of $t\mapsto (g(u_0+t)-g(u_0))$ shows that
\begin{equation}\label{506}
\varepsilon_n:=\sup_{x\in\R^N}|g(u_0(x)+h_n(x))-g(u_0(x))|\to 0,\quad\mbox{as } n\to \infty.
\end{equation}
Consider
\begin{equation}\label{507}
v_n\in D^{s,2}(\R^N): (-\Delta)^sv_n= \hat{a}(x)\varepsilon_n,
\end{equation}
and
\begin{equation}\label{508}
z_n\in D^{s,2}(\R^N): (-\Delta)^sz_n= -\hat{a}(x)\varepsilon_n.
\end{equation}
By comparison we see that $z_n\le h_n\le v_n$, and applying Theorem \ref{T401} along with Corollary \ref{C401} with $p=2$, we obtain
$$
v_n(x)\le C\frac{\varepsilon_n}{1+|x|^{N-2s}}\quad\mbox{and } z_n(x)\ge -C\frac{\varepsilon_n}{1+|x|^{N-2s}},
$$
which shows that
$$
|h_n(x)|\left(1+|x|^{N-2s}\right)\le C\,\varepsilon_n \to 0, \quad\mbox{uniformly, as } n\to \infty,
$$
that is, $\|h_n\|_{V_s}\to 0$, since $h_n\to 0$ in $D^{s,2}(\R^N)$. Finally, since $u_0$ is a local minimizer of $\Phi$ in the $V_s$-topology we get with $h_n\to 0$ in $V_s$ for $n$ large
$$
\Phi(u_0) \le \Phi(u_0+h_n)=\Phi(u_n)=\inf_{h\in B_n} \Phi(u_0+h),
$$
where
$$
B_n=\left\{h\in D^{s,2}(\R^N): [h_n]_{s,2}\le \frac{1}{n}\right\},
$$
which proves that $u_0$ must be a local minimizer of $\Phi$ in the $[\cdot]_{s,2}$-topology completing the proof of Theorem \ref{T501}. \hfill $\Box$

\section*{Declarations}
\begin{itemize}
\item Funding:  No funding was received to assist with the preparation of this manuscript.
\item Conflict of interest: The authors have no conflicts of interest to declare that are relevant to the
content of this article.
\item Availability of data and materials: The authors declare that the data supporting the findings of this study are
available within the paper.
\end{itemize}

%%%%%%%%%%%%%%%%%%%%%%%%%%%%%%%%%%%%%%%%%%%%%%%%%%%%

%%%%%%%%%%%%%thebibliography%%%%%%%%%%%%%%%%%%%%%%%%%%%%%%%%%%%

\end{document}